\documentclass[12pt,reqno]{amsart}

\usepackage{amsmath,amssymb,amsthm,mathtools}
\usepackage[colorlinks=true,allcolors=blue,hypertexnames=false]{hyperref}
\usepackage{microtype}
\hypersetup{pdftitle={Algebraically Integrable Smooth Convex Bodies},pdfauthor={Junjie Shan}}

\numberwithin{equation}{section}

\newtheorem{theorem}{Theorem}[section]
\newtheorem{proposition}[theorem]{Proposition}
\newtheorem{lemma}[theorem]{Lemma}
\newtheorem{corollary}[theorem]{Corollary}
\theoremstyle{definition}
\newtheorem{definition}[theorem]{Definition}
\theoremstyle{remark}
\newtheorem{remark}[theorem]{Remark}

\newcommand{\R}{\mathbb R}
\newcommand{\C}{\mathbb C}
\newcommand{\Q}{\mathbb Q}
\newcommand{\Pj}{\mathbb P}
\newcommand{\Mcal}{\mathcal M}
\newcommand{\Dcal}{\mathcal D}
\newcommand{\Zcal}{\mathcal Z}
\newcommand{\Sph}{S^{n-1}}
\DeclareMathOperator{\vol}{vol}
\DeclareMathOperator{\Sing}{Sing}
\DeclareMathOperator{\Span}{span}
\DeclareMathOperator{\Res}{Res}

\title[Algebraically integrable convex bodies]{Algebraically Integrable Smooth Convex Bodies}
\author{Junjie Shan}
\address{School of Mathematics, Sichuan University, Chengdu 610064, Sichuan, China}
\email{shanjjmath@163.com}
\date{}

\begin{document}

\begin{abstract}
We prove that a convex body with $C^\infty$ boundary in an
odd dimension $n\ge3$ is algebraically integrable if and only
if it is an ellipsoid. This answers Arnold's problem for
smooth convex bodies. Together with Vassiliev's
even-dimensional nonexistence theorem, it gives the
classification in every dimension $n\ge2$.
More generally, in odd dimensions, a convex body with
$C^\infty$ semialgebraic boundary is an ellipsoid whenever
its section volume function satisfies a joint algebraic relation.
\end{abstract}
\maketitle

\section{Introduction}

In \emph{Philosophiae Naturalis Principia Mathematica} (1687),
Newton studied the algebraic quadrature of ovals in connection
with celestial mechanics. His lemma on ovals
\cite[Book~I, Section~VI, Lemma~XXVIII]{Newton} states, in modern
terms, that the areas cut off from a smooth planar convex domain
by lines do not depend algebraically on the line parameters.
The higher-dimensional problem is to characterize bodies whose
volumes cut off by hyperplanes depend algebraically on the
hyperplane parameters. Such bodies are called \emph{algebraically
	integrable}. Ellipsoids in odd dimensions have this property.
Arnold \cite[Problem~1990-27]{ArnoldProblems} asked whether they
are the only smooth convex bodies with this property.
For background and related questions, see
\cite{ABKVY2023,Vassiliev2002}.

For a convex body $K\subset\R^n$ and $\xi\in\Sph$, set
\begin{equation}\label{eq:support-interval}
	a_K(\xi)=\min_{x\in K}x\cdot\xi,
	\qquad b_K(\xi)=\max_{x\in K}x\cdot\xi,
	\qquad I_K(\xi)=(a_K(\xi),b_K(\xi)).
\end{equation}
The cutoff volume and the section volume are
\begin{align}
 V_K(\xi,t)&=\vol_n\bigl(K\cap\{x:x\cdot\xi\le t\}\bigr),\label{eq:cutoff}\\
 A_K(\xi,t)&=\vol_{n-1}\bigl(K\cap\{x:x\cdot\xi=t\}\bigr),\label{eq:section}
\end{align}
respectively, and satisfy $\partial_tV_K(\xi,t)=A_K(\xi,t)$ for $t\in I_K(\xi)$. Algebraic integrability concerns the cutoff volume, rather than the section volume. We use nonnormalized hyperplane parameters in the following definition.

\begin{definition}\label{def:integrable}
	A convex body $K\subset\R^n$ is \emph{algebraically integrable}
	if there is a nonzero polynomial
	$P(\eta,s,w)\in\R[\eta_1,\ldots,\eta_n,s,w]$ such that
	\begin{equation}\label{eq:cutoff-algebraic}
		P\bigl(\eta,s,\vol_n(K\cap\{x:\eta\cdot x\le s\})\bigr)=0
	\end{equation}
	whenever $\eta\in\R^n\setminus\{0\}$ and the hyperplane
	$\eta\cdot x=s$ intersects the interior of $K$.
\end{definition}

A convex body $K$ is called \emph{polynomially integrable} if
$A_K(\xi,\cdot)$ agrees with a polynomial on $I_K(\xi)$ for
every $\xi\in\Sph$, with a degree bound independent of $\xi$.
For convex bodies with $C^\infty$ boundary in odd dimensions,
Agranovsky \cite{Agranovsky2018} obtained the ellipsoid
characterization under the degree bound $n-1$.
Koldobsky, Merkurjev, and Yaskin \cite{KMY} removed this
degree restriction.
For algebraic cutoff volumes, Agranovsky
\cite[Definition~1.4 and Theorem~1.5]{Agranovsky2019}
obtained the same characterization under his additional
condition that the cutoff volume function be
\emph{free of real singularities}.

Vassiliev \cite{Vassiliev2019} constructed smooth nonconvex
domains whose cutoff volumes depend algebraically on the
hyperplane parameters on certain open sets of cutting
hyperplanes, rather than requiring this property for all cuts.
He also constructed nonconvex domains whose cutoff volume
functions have finitely branched analytic continuations
\cite{Vassiliev2020}; their algebraicity was conjectured there.
Further obstructions to algebraic integrability arising from
ramification near asymptotic hyperplanes were obtained by
Artemov \cite{Artemov2026}.
For the boundary of a convex body, smooth always means $C^\infty$.

In this paper, we give a complete solution to Arnold's problem
on algebraically integrable smooth convex bodies.
Our main theorem establishes the ellipsoid characterization
in every odd dimension, without any additional restrictions
on the degree or form of the defining algebraic relation.

\begin{theorem}\label{thm:main}
	Let $n\ge3$ be odd, and let $K\subset\R^n$ be a convex body
	with $C^\infty$ boundary. Then $K$ is algebraically integrable
	if and only if $K$ is an ellipsoid.
\end{theorem}

Vassiliev \cite[Theorem~1]{Vassiliev2015} proved that bounded
domains with smooth boundaries in even-dimensional Euclidean
spaces are not algebraically integrable.
Together with Theorem~\ref{thm:main}, this yields a complete
classification of algebraically integrable smooth convex bodies
in all dimensions.

\begin{corollary}\label{cor:all-dimensions}
	Let $n\ge2$, and let $K\subset\R^n$ be a convex body
	with $C^\infty$ boundary. Then $K$ is algebraically integrable
	if and only if $n$ is odd and $K$ is an ellipsoid.
\end{corollary}

The smoothness assumption cannot be omitted: every full-dimensional convex polytope is algebraically integrable, as follows from the simplex volume formula in \cite[Theorem~2.2]{Lasserre} and triangulation.

We also obtain a characterization in terms of section volumes. Agranovsky \cite[Theorem~1.1]{Agranovsky2026} characterized smooth strictly convex bodies whose section functions satisfy a nontrivial equation of the form
\[
 q(\xi)A_K(\xi,t)^r+p(\xi,t)=0,
\]
where $r$ is a positive integer and $p,q$ are polynomials, with $q$ independent of $t$. In odd dimensions, we allow an arbitrary joint algebraic relation, under a semialgebraic boundary assumption.

A subset of a real Euclidean space is \emph{semialgebraic} if it is a finite union of sets defined by finitely many polynomial equalities and inequalities. We say that $A_K$ satisfies a \emph{joint algebraic relation} if there is a real polynomial $Q(\xi,t,w)$, not identically zero on $\Sph\times\R\times\R$, such that
\begin{equation}\label{eq:section-algebraic}
 Q\bigl(\xi,t,A_K(\xi,t)\bigr)=0
 \qquad \bigl(\xi\in\Sph,\ t\in I_K(\xi)\bigr).
\end{equation}
The nonvanishing condition excludes identities coming solely from $\sum_i\xi_i^2=1$. The coefficient of each power of $A_K$ may depend on both the direction and the intercept.

\begin{theorem}\label{thm:section}
Let $n\ge3$ be odd, and let $K\subset\R^n$ be a convex body with $C^\infty$ semialgebraic boundary. If $A_K$ satisfies a joint algebraic relation, then $K$ is an ellipsoid.
\end{theorem}

Lemma~\ref{lem:algebraization} shows that algebraic integrability supplies the hypotheses of Theorem~\ref{thm:section}. No assumption is made on the singularities of the complex algebraic closure of the boundary, and positive curvature is not required everywhere on the real boundary.

The proof combines algebraic geometry with the topology of complex
sections. Near a general support hyperplane, the real section boundary
is a vanishing cycle associated with a unique nondegenerate critical
point. The section volume is a nonzero period of this cycle, and the
Picard--Lefschetz formula describes its monodromy. Following the intersection-theoretic obstruction in Vassiliev \cite{Vassiliev2018}, we combine finite branching with Deligne's semisimplicity theorem \cite{Deligne1971}. Proposition~\ref{prop:semisimplicity} shows that the cycle maps to zero in the middle homology of a smooth projective compactification.

The original cycle need not be zero in the open fiber. Our compactification is chosen so that, for a general direction, its boundary remains regular over every finite complex parameter. We can then represent the cycle near the divisor at infinity. Proposition~\ref{prop:boundary-entire} extends its period to an entire function. The joint algebraic equation gives a polynomial degree bound independent of the direction, and the characterization of polynomially integrable bodies in \cite{KMY} yields the conclusion.

\section{Preliminaries}\label{sec:preliminaries}

\subsection{Notation and terminology}
Throughout the paper, $n\ge2$. We use $x\cdot y$ and $|x|$ for the Euclidean inner product and norm, and write $S^{n-1}=\{\xi\in\R^n:|\xi|=1\}$.  When used with complex coordinates, the dot product denotes
the complex bilinear extension
$z\cdot w=\sum_i z_iw_i$, without complex conjugation. 
For $j\ge1$, $\vol_j$ denotes $j$-dimensional Euclidean volume in the relevant affine subspace. We write $B_2^j$ for the Euclidean unit ball in $\R^j$ and $\omega_j=\vol_j(B_2^j)$. A convex body is a compact convex set $K\subset\R^n$ with
nonempty interior. We denote the interior of $K$ by
$\operatorname{int}K$. Background on convex bodies and their sections can be found in \cite{Gardner,Schneider}.

We use the classical language of algebraic varieties. For background
on complex algebraic geometry, see Harris \cite{Harris} and
Hartshorne \cite{Hartshorne}; for real algebraic and semialgebraic
geometry, see Bochnak, Coste, and Roy \cite{BCR}. Unless real points
are explicitly specified, algebraic varieties are over $\C$.

An \emph{affine algebraic set} $Y\subset\C^N$ is the common zero
set of finitely many complex polynomials. These are the closed sets
of the \emph{Zariski topology} on $\C^N$. A subset $Z\subset Y$ is
\emph{algebraic} if it is the common zero set in $Y$ of finitely many
polynomials on $\C^N$. It is proper if $Z\ne Y$.

Complex projective space $\Pj^N$ consists of the classes
$[z_0:\cdots:z_N]$ of nonzero vectors in $\C^{N+1}$ under
multiplication by a common nonzero complex scalar. The standard
affine charts $\{z_j\ne0\}$ are identified with $\C^N$ by the ratios
$z_i/z_j$, $i\ne j$. A \emph{projective algebraic set} is the common
zero set in $\Pj^N$ of finitely many homogeneous polynomials. These
sets define the closed sets of the Zariski topology on $\Pj^N$.
We call affine and projective algebraic sets and their Zariski open
subsets \emph{algebraic varieties}.

Each variety carries the induced Zariski topology. An
\emph{algebraic subset} of a variety means a Zariski closed subset
of that variety.  A \emph{locally closed} subset is the intersection
of a Zariski open and a Zariski closed subset of the same variety.

A variety $Y$ is \emph{irreducible} if it is nonempty and cannot
be written as the union of two proper Zariski closed subsets
of $Y$. The \emph{irreducible components} of a variety $Y$ are its
maximal irreducible Zariski closed subsets. There are finitely
many such components, and their union is $Y$. The \emph{Zariski
	closure} of a subset $S\subset Y$ is the smallest Zariski closed
subset of $Y$ containing $S$. The subset $S$ is \emph{Zariski dense}
in $Y$ if this closure equals $Y$. When parameters range over an
irreducible algebraic variety $Y$, we say that a property holds for
a \emph{general} parameter if there is a proper Zariski closed
subset $B\subsetneq Y$ such that the property holds for every
parameter in $Y\setminus B$.

A point $p$ of a complex variety $Y$ is \emph{smooth} if $Y$ is a
complex submanifold near $p$. The smooth locus is denoted by
$Y_{\mathrm{reg}}$, and the \emph{singular locus}
$\Sing Y=Y\setminus Y_{\mathrm{reg}}$ is Zariski closed in $Y$.
The dimension of a nonempty complex variety is the maximum of the
complex manifold dimensions of the smooth loci of its irreducible
components. Every nonempty proper algebraic subset of an
irreducible variety has smaller dimension. 

A nonconstant polynomial $F\in\C[z_1,\ldots,z_N]$ is
\emph{irreducible} if it cannot be written as the product
of two nonconstant polynomials in $\C[z_1,\ldots,z_N]$.
For such a polynomial, the zero set
$X=\{z\in\C^N:F(z)=0\}$ has complex dimension $N-1$.
A point $p\in X$ is smooth if and only if $dF_p\ne0$,
that is, at least one of the partial derivatives
$\partial F/\partial z_j$ is nonzero at $p$.

A function on a variety is \emph{regular} if, in affine charts,
it is locally a quotient of polynomials whose denominator is
nowhere zero on the chosen Zariski neighborhood. Let $X$ and $Y$
be algebraic varieties. A map $f:X\to Y$ is a \emph{morphism}
if, locally in the Zariski topology, it is given by regular
coordinate functions in affine charts on the source and target.
An \emph{isomorphism} is a bijective morphism whose inverse
is also a morphism. For an irreducible affine variety $Y$, a rational
function is represented by $p/q$, where $p$ and $q$ are polynomial
functions on $Y$ and $q$ is not identically zero on $Y$.

Let $Y$ and $T$ be algebraic varieties, and let $f:Y\to T$ be a
morphism. For $t\in T$, its \emph{fiber} over $t$ is
$f^{-1}(t)=\{p\in Y:f(p)=t\}$. If $Y$ and $T$ are smooth,
the differential at $p\in Y$ is
$df_p:T_pY\to T_{f(p)}T$. The morphism is \emph{smooth} if $df_p$ is surjective for
every $p\in Y$; equivalently, its associated holomorphic map is a
\emph{submersion}. In this setting, a point where $df_p$ is not
surjective is a \emph{critical point}, and its image is a
\emph{critical value}.

A variety is \emph{projective} if it is isomorphic to a Zariski
closed subset of some $\Pj^N$. Let $Y$ and $T$ be algebraic
varieties. A morphism $f:Y\to T$ is \emph{projective} if there exist
an integer $N\ge0$ and an isomorphism $\iota$ from $Y$ onto a
Zariski closed subset of $T\times\Pj^N$ such that
$f=\operatorname{pr}_T\circ\iota$, where
$\operatorname{pr}_T(t,[z])=t$.

For compactness, homology, homotopy, and analytic continuation, we
use the usual topology induced by the Euclidean topology in affine
charts, not the Zariski topology. A projective morphism is
\emph{proper}: inverse images of compact sets are compact.
In particular, projective varieties and the fibers of projective
morphisms are compact. A \emph{compactification} of a variety $Y$
is a compact algebraic variety containing $Y$ as a Zariski open
dense subset.

For a subset $S\subset\R^N$, its real Zariski closure is the
common zero set in $\R^N$ of all real polynomials vanishing on $S$.
If $S$ is nonempty and semialgebraic, it admits a finite partition
into smooth semialgebraic submanifolds. Its dimension is the largest
of their real dimensions. This equals
the dimension of its real Zariski closure
\cite[Proposition~2.8.2]{BCR}. Projections of semialgebraic sets are
semialgebraic by the Tarski--Seidenberg theorem
\cite[Theorem~2.2.1]{BCR}. A map between real semialgebraic sets is
\emph{semialgebraic} if its graph is semialgebraic. A real analytic
submanifold $M\subset\R^N$ is a \emph{Nash submanifold} if $M$
is semialgebraic. A map between Nash submanifolds is
\emph{Nash} if it is real analytic and semialgebraic.
Lemma~\ref{lem:nash-boundary} recalls that if a convex body
$K$ has $C^\infty$ semialgebraic boundary, then $\partial K$
is real analytic.

For a topological space $Y$, we write $H_j(Y;\Q)$ and $H^j(Y;\Q)$
for singular homology and cohomology with rational coefficients.
A cycle is a finite singular chain with zero boundary. In local
holomorphic coordinates $z_1,\ldots,z_m$, a holomorphic $k$-form
is a sum of terms $a(z)\,dz_{i_1}\wedge\cdots\wedge dz_{i_k}$
with holomorphic coefficients $a$. We integrate such forms over
real $k$-dimensional cycles, using piecewise smooth representatives.
 We use the Mayer--Vietoris sequence for an open
cover and the intersection pairing supplied by Poincar\'e duality
on a compact oriented manifold without boundary; see \cite{Hatcher}.

\subsection{Algebraic boundary and section functions}

Let $P(\eta,s,w)$ be the polynomial in
Definition~\ref{def:integrable}, where $w$ is the volume
variable. For a hyperplane $\eta\cdot x=s$ meeting
$\operatorname{int}K$, the two cutoff volumes
\[
\vol_n\bigl(K\cap\{x:\eta\cdot x\le s\}\bigr),
\qquad
\vol_n\bigl(K\cap\{x:\eta\cdot x\ge s\}\bigr)
\]
satisfy $P(\eta,s,w)=0$ and $P(-\eta,-s,w)=0$,
respectively. Hence the nonzero polynomial
\[
w(w-\vol_nK)P(\eta,s,w)P(-\eta,-s,w)
\]
vanishes at both volumes. For a hyperplane not meeting
$\operatorname{int}K$, the two volumes are $0$ and
$\vol_nK$, so the factor $w(w-\vol_nK)$ gives the same
conclusion. Thus one polynomial equation holds for both
cutoff volumes of every hyperplane. Conversely, such an
equation satisfies Definition~\ref{def:integrable}.

The next two lemmas pass from the volume relation to a complex algebraic hypersurface containing the real boundary.

\begin{lemma}\label{lem:algebraization}
If a convex body $K$ with $C^\infty$ boundary is algebraically integrable, then $K$ and $\partial K$ are semialgebraic. Moreover, $A_K$ satisfies a joint algebraic relation.
\end{lemma}

\begin{proof}
	The boundary is semialgebraic by
	\cite[Remark~1(2)]{Vassiliev2015}.
	Every point of $K$ lies on a segment with endpoints in
	$\partial K$, and every such segment is contained in $K$.
	Thus $K$ is the image of the semialgebraic set
	$\partial K\times\partial K\times[0,1]$ under the polynomial
	map $(x,y,u)\mapsto(1-u)x+uy$. It is therefore semialgebraic
	by \cite[Proposition~2.2.7]{BCR}.
	
	It remains to deduce from the algebraic relation for $V_K$
	a polynomial relation for its derivative $A_K=\partial_tV_K$.
	Let $P$ satisfy \eqref{eq:cutoff-algebraic}. Although $P$
	is nonzero, its restriction to $|\eta|=1$ might be
	identically zero. We  choose another polynomial relation
	to avoid this.
	
	Expand
	\[
	P(c\eta,cs,w)=\sum_{j=0}^D c^jP_j(\eta,s,w)
	\]
	as a polynomial in $c$. For $\xi\in\Sph$ and
	$t\in I_K(\xi)$, multiplying the inequality
	$x\cdot\xi\le t$ by $c>0$ leaves the cutoff volume
	unchanged. Thus
	\[
	\sum_{j=0}^D c^jP_j(\xi,t,V_K(\xi,t))=0
	\qquad\text{for all }c>0.
	\]
	Each coefficient is therefore zero, so
	$P_j(\xi,t,V_K(\xi,t))=0$ for every $j$.
	At least one $P_j$ is not identically zero on
	$\Sph\times\R\times\R$. Otherwise,
	$P(c\xi,ct,w)=0$ for all $c>0$, $\xi\in\Sph$,
	and $t,w\in\R$, contradicting $P\ne0$. We replace $P$ by such a $P_j$.
	
	We regard $P(\xi,t,w)$ as a polynomial in $w$. Its
	coefficients are polynomial functions of $(\xi,t)$ restricted
	to $\Sph\times\R$. We work over the fraction field
	\[
	\mathbb F
	=\operatorname{Frac}\!\left(
	\R[\xi_1,\ldots,\xi_n,t]\Big/
	\Big(\sum_{i=1}^n\xi_i^2-1\Big)
	\right).
	\]
	The quotient ring is an integral domain, since
	$\sum_{i=1}^n\xi_i^2-1$ is irreducible for $n\ge2$.
	Thus $\mathbb F$ consists of ratios of polynomial functions
	on $\Sph\times\R$.
	
	The degree of $P$ in $w$ is positive. Otherwise, for each
	fixed $\xi\in\Sph$, the polynomial $P(\xi,t)$ in $t$
	would vanish on the nonempty interval $I_K(\xi)$ and hence
	for every $t\in\R$. This contradicts the choice of $P$.
	
	Factor $P$ in $\mathbb F[w]$ as
	\[
	P=a\,f_1^{m_1}\cdots f_r^{m_r},
	\]
	where $a\in\mathbb F\setminus\{0\}$, the integers $m_i$
	are positive, and $f_1,\ldots,f_r$ are distinct monic
	irreducible polynomials of positive degree in $w$. Define
	\[
	B(\xi,t,w)=f_1(\xi,t,w)\cdots f_r(\xi,t,w).
	\]
	Thus $B$ is monic and square-free.
	
	Let $E\subset\Sph\times\R$ be the union of the zero sets
	of the denominators in $a$ and in the coefficients of the
	$f_i$, together with the zero set of the numerator of $a$.
	None of these polynomial functions vanishes identically on
	$\Sph\times\R$, so $E$ is a proper algebraic subset.
	For $(\xi,t)\notin E$ with $t\in I_K(\xi)$, the identity
	\[
	0=P(\xi,t,V_K(\xi,t))
	=a(\xi,t)\prod_{i=1}^r
	f_i(\xi,t,V_K(\xi,t))^{m_i}
	\]
	implies $B(\xi,t,V_K(\xi,t))=0$.
	This identity holds on an open set of internal parameters.
	Differentiating with respect to $t$, with $\xi$ fixed, gives
	\[
	B_w(\xi,t,V_K(\xi,t))A_K(\xi,t)
	+B_t(\xi,t,V_K(\xi,t))=0.
	\]
	
	Introduce an independent variable $z$ and define
	\[
	Q(\xi,t,z)
	=\Res_w\!\left(B,B_wz+B_t\right)\in\mathbb F[z],
	\]
	where $\Res_w$ denotes the resultant with respect to $w$.
	The coefficients of $B$, $B_w$, and $B_t$ are defined
	outside $E$, and the resultant is a polynomial in these
	coefficients. For $z=A_K(\xi,t)$, its two arguments have
	the common root $w=V_K(\xi,t)$. Hence
	\[
	Q(\xi,t,A_K(\xi,t))=0
	\]
	for $(\xi,t)\notin E$ with $t\in I_K(\xi)$.
	
	This relation is nontrivial. Since $\mathbb F$ has
	characteristic zero and $B$ is square-free, $B$ and $B_w$
	are relatively prime. Therefore $\Res_w(B,B_w)\ne0$.
	If $d=\deg_w B$, the coefficient of $z^d$ in $Q$ is
	$\Res_w(B,B_w)$, because $B$ is monic.
	Thus $Q\ne0$ in $\mathbb F[z]$.
	
	Multiplying $Q$ by a common denominator of its coefficients
	gives a real polynomial in $(\xi,t,z)$ whose restriction
	to $\Sph\times\R\times\R$ is not identically zero.
	This polynomial vanishes at $z=A_K(\xi,t)$ for internal
	parameters outside $E$.
	The set $E$ has empty interior, since a nonzero polynomial
	function on the connected real analytic manifold
	$\Sph\times\R$ cannot vanish on a nonempty open set.
	Its complement is therefore dense.
	Every internal hyperplane meets $\partial K$ transversely,
	since a tangent hyperplane of a smooth convex body is
	supporting. Consequently, $A_K$ is continuous on the open
	set of internal parameters, and the polynomial relation
	extends to all $\xi\in\Sph$ and $t\in I_K(\xi)$.
	This proves the last assertion.
\end{proof}

\begin{lemma}\label{lem:nash-boundary}
Let $K\subset\R^n$ be a convex body with $C^\infty$ semialgebraic boundary. Then $\partial K$ is real analytic. There is a real polynomial $F$, irreducible over $\C$, such that
\[
 \partial K\subset X:=\{F=0\}\subset\C^n,
\]
and $\partial K$ is Zariski dense in $X$.
\end{lemma}

\begin{proof}
Local graph representations of $\partial K$ are $C^\infty$
semialgebraic functions, hence real analytic by
\cite[Definition~2.9.3 and Proposition~8.1.8]{BCR}.
Thus $\partial K$ is a connected Nash hypersurface.
Its real Zariski closure is irreducible of dimension $n-1$
by \cite[Proposition~8.4.1(i)]{BCR}.
The ideal $I\subset\R[x_1,\ldots,x_n]$ of polynomials
vanishing on $\partial K$ is therefore prime of height one.
Since $\R[x_1,\ldots,x_n]$ is a unique factorization domain,
$I=(F)$ for a real irreducible polynomial $F$.
	
For $G\in\C[x_1,\ldots,x_n]$, write $G=G_1+iG_2$
with $G_1,G_2\in\R[x_1,\ldots,x_n]$.
If $G$ vanishes on $\partial K$, then both $G_1$ and $G_2$
vanish there and are therefore divisible by $F$.
Hence $F$ divides $G$ in $\C[x_1,\ldots,x_n]$, so the ideal
of complex polynomials vanishing on $\partial K$ is $(F)$.
	
	This complex ideal is prime. To see this, suppose that
	$GH$ vanishes on $\partial K$ and that $G$ does not
	vanish identically there. Then $G\ne0$ on a nonempty
	relatively open subset of $\partial K$, so $H=0$ on
	that subset. The real analytic identity theorem on the
	connected manifold $\partial K$ gives $H=0$ on all of
	$\partial K$. Thus $(F)$ is prime over $\C$, and $F$
	is irreducible over $\C$.
Every complex polynomial vanishing on $\partial K$ is
divisible by $F$ and hence vanishes on $X$.
Thus $\partial K$ is Zariski dense in $X$.
\end{proof}

The hypersurface $X$ in Lemma~\ref{lem:nash-boundary} is the smallest complex algebraic set containing $\partial K$. The real points of $X$ are precisely the real solutions of
$F=0$. They contain $\partial K$, but need not coincide
with it. For the unit ball, $X$ is the complex quadric
$z_1^2+\cdots+z_n^2=1$, whose real points form the unit
sphere, although $X$ also contains nonreal points.
In general, the smoothness of $\partial K$ does not imply
that the complex hypersurface $X$ is nonsingular.

\begin{lemma}\label{lem:real-analytic-sections}
For a convex body $K\subset\R^n$ with real analytic boundary and a fixed $\xi\in\Sph$, the function $A_K(\xi,\cdot)$ is real analytic on $I_K(\xi)$. If $\partial K$ is only $C^\infty$ and $0\in\operatorname{int}K$, then $A_K(\xi,t)$ is $C^\infty$ as a function of $(\xi,t)$
in a neighborhood of $\Sph\times\{0\}$.
\end{lemma}

\begin{proof}
A hyperplane meeting $\operatorname{int}K$
	cannot be tangent to $\partial K$. Equivalently, the
	differential of $x\mapsto x\cdot\xi$ is nonzero on
	$T_x\partial K$ at every point of an internal section
	boundary. Thus all internal sections are transverse.
	
	First assume that $\partial K$ is real analytic. Fix
	$\xi\in\Sph$ and $t_0\in I_K(\xi)$, and choose orthonormal
	coordinates with $\xi=e_n$. Put
	\[
	C_t=\partial K\cap\{x_n=t\}.
	\]
	Let $\nabla_{\partial K}$ denote the tangential gradient
	for the induced Euclidean metric. On the open subset $\{x\in\partial K:x_n\in I_K(\xi)\}$, the vector field
	\[
	Y=\frac{\nabla_{\partial K}x_n}
	{|\nabla_{\partial K}x_n|^2}
	\]
	is well defined and real analytic, and satisfies $dx_n(Y)=1$.
	Let $\Phi_s$ be the local flow of $Y$, characterized by
	$\partial_s\Phi_s(p)=Y(\Phi_s(p))$ and $\Phi_0(p)=p$.
	In local analytic coordinates, the coefficients of $Y$
	extend holomorphically to a complex neighborhood.
	The Picard iteration argument in
	\cite[proof of Theorem~4.2]{Teschl}, applied to the resulting
	system, shows that $(s,p)\mapsto\Phi_s(p)$ is real analytic.
	
	Choose a closed interval about $t_0$ contained in $I_K(\xi)$.
	The part of $\partial K$ over this interval is compact,
	so the flow is defined there for uniformly small positive
	and negative times. Since $dx_n(Y)=1$,
	\(
	\frac{d}{ds}x_n(\Phi_s(x))=1,
	\)
	and hence $x_n(\Phi_s(x))=x_n(x)+s$.
	Therefore $\Phi_s$ maps $C_{t_0}$ onto $C_{t_0+s}$,
	with inverse $\Phi_{-s}$, for sufficiently small $|s|$.
	
	Orient each section hyperplane by
	$dx_1\wedge\cdots\wedge dx_{n-1}$, and give $C_t$ the
	induced boundary orientation. Consider the polynomial
	$(n-2)$-form
	\[
	\beta=x_1\,dx_2\wedge\cdots\wedge dx_{n-1},
	\qquad
	d\beta=dx_1\wedge\cdots\wedge dx_{n-1},
	\]
	with $\beta=x_1$ when $n=2$.
	Stokes' theorem \cite[Theorem~16.11]{Lee} gives
	\[
	A_K(\xi,t)
	=\int_{C_t}\beta
	=\int_{C_{t_0}}\Phi_{t-t_0}^{*}\beta
	\]
	for $t$ sufficiently close to $t_0$; the flow preserves
	the boundary orientations for small times.
	
	This last integral is over a fixed compact manifold.
	Choose a finite analytic coordinate cover of $C_{t_0}$
	and a subordinate smooth partition of unity, independent
	of $t$, with compact supports in the coordinate charts.
	After shrinking the neighborhoods, the coefficients of
	the pulled-back forms extend holomorphically in $t$ to
	a common complex disk about $t_0$ and are uniformly bounded
	on the supports of the partition of unity.
	Integration in the fixed real coordinates therefore gives
	a holomorphic function of $t$. Its restriction to real $t$
	is $A_K(\xi,t)$, proving real analyticity near $t_0$.
	Since $t_0$ was arbitrary, this holds on $I_K(\xi)$.
	
	Now assume only that $\partial K$ is $C^\infty$ and that
	$0\in\operatorname{int}K$. Choose $\varepsilon>0$ with
	$\varepsilon B_2^n\subset\operatorname{int}K$, and set
	\[
	\mathcal E=
	\{(x,\xi,t)\in\partial K\times\Sph\times(-\varepsilon,\varepsilon):
	x\cdot\xi=t\}.
	\]
Since $\partial_t(x\cdot\xi-t)=-1$, $\mathcal E$ is a
smooth manifold. Consider the projection
	$\pi(x,\xi,t)=(\xi,t)$.
	Fix $(x,\xi,t)\in\mathcal E$ and take any
	$(v,a)\in T_\xi\Sph\times\R$.
	A vector $(h,v,a)$ is tangent to $\mathcal E$ precisely
	when $h\in T_x\partial K$ and
	$\xi\cdot h+x\cdot v=a$.
	By transversality, the linear map
	$h\mapsto\xi\cdot h$ from $T_x\partial K$ to $\R$
	is nonzero and hence surjective. Thus such an $h$ exists
	for every $(v,a)$.
	Since $d\pi_{(x,\xi,t)}(h,v,a)=(v,a)$, the differential
	of $\pi$ is surjective. Hence $\pi$ is a submersion.
	It is surjective because $|t|<\varepsilon$ gives an internal
	section, and proper because $\partial K$ is compact.
	Ehresmann's theorem \cite[Lemma~9.2]{KolarMichorSlovak}
	therefore supplies smooth local trivializations of this
	family of section boundaries.

	Locally in $\xi$, choose smoothly varying orthonormal
	coordinates on $\xi^\perp$, translate them to
	$t\xi+\xi^\perp$, and define $\beta_\xi$ as above.
	Over a sufficiently small parameter neighborhood, the
	preceding trivialization can be chosen to give a smooth
	family of orientation-preserving diffeomorphisms
	$\psi_{\xi,t}:C_*\to\partial K\cap\{x:x\cdot\xi=t\}$,
	where $C_*$ is a fixed section boundary.
	By Stokes' theorem and change of variables,
	$A_K(\xi,t)=\int_{C_*}\psi_{\xi,t}^*\beta_\xi$.
	The pulled-back form depends smoothly on $(\xi,t)$, and
	compactness of $C_*$ allows repeated differentiation
	under the integral. Since this argument applies near every
	parameter, $A_K$ is $C^\infty$ on
	$\Sph\times(-\varepsilon,\varepsilon)$.
\end{proof}

\section{A resolved parallel pencil}\label{sec:pencil}

For the odd-dimensional part of the proof, let $m\ge1$ be an
integer and write
\begin{equation}\label{eq:dimensions}
	n=2m+1,\qquad d=n-2=2m-1.
\end{equation}
Let $K\subset\R^n$ be a convex body with $C^\infty$
semialgebraic boundary, and let $F$ and $X=\{F=0\}$ be as in
Lemma~\ref{lem:nash-boundary}. An internal real section of $K$
has dimension $d+1$, and its boundary has real dimension $d$.
A smooth complex hyperplane section of $X$ has complex
dimension $d$, hence real dimension $2d$. This explains the
homology degree $d$ used below.

Since a nonzero polynomial cannot vanish on an open subset of
$\R^n$, there is a point $c\in\operatorname{int}K$ with
$F(c)\ne0$. Translation by $-c$ replaces $t$ with
$t+\xi\cdot c$ in the cutoff and section functions.
It therefore preserves algebraic integrability and the joint
algebraic relation \eqref{eq:section-algebraic}.
Translating $K$ and $X$ together and keeping the same notation,
we may assume
\begin{equation}\label{eq:origin}
	0\in\operatorname{int}K,\qquad 0\notin X.
\end{equation}
For a fixed direction $\xi$, the hyperplanes $\xi\cdot x=t$
form a \emph{parallel pencil}. We construct a smooth total
space for the corresponding sections after resolution.
The fibers may still have singularities in the affine part.
The essential point is that the map is smooth near the added
boundary for every finite complex parameter, as required in
Proposition~\ref{prop:boundary-entire}.

\subsection{Resolution and the boundary at infinity}

A morphism between irreducible varieties is \emph{birational}
if it restricts to an isomorphism between dense open subsets.
We use a \emph{resolution} in the form of a projective
birational morphism from a smooth irreducible variety.
A reduced divisor on a smooth variety is a hypersurface with
each irreducible component counted once; the notation
$(\ )_{\mathrm{red}}$ removes multiplicities. A divisor has
\emph{simple normal crossings} if its components are smooth
and, in local holomorphic coordinates, their union is given
by $u_1\cdots u_r=0$. Its smooth locally closed strata are obtained by intersecting
chosen components and removing the points lying on the others.
We also use algebraic stratifications adapted to morphisms;
see \cite{GoreskyMacPherson}.

Embed $\C^n$ as the chart $z_0\ne0$ in $\Pj^n$, with affine
coordinates $x_i=z_i/z_0$. Let $\overline X\subset\Pj^n$ be
the projective closure of $X$, and put $H_\infty=\{z_0=0\}$.
First resolve $\overline X$, and then take an embedded
resolution of the inverse image of $H_\infty$ in the
resulting smooth variety.
By \cite[Theorems~1.1 and~1.3]{BierstoneMilman}, this gives a
projective birational morphism
\[
\pi:M\longrightarrow\overline X
\]
with $M$ smooth irreducible projective of dimension $d+1$ and
\[
D_\infty=(\pi^{-1}H_\infty)_{\mathrm{red}}
\]
a simple normal-crossing divisor. The first resolution is an
isomorphism over the smooth locus of $\overline X$. The second
uses blowings-up with centers over $H_\infty$, so $\pi$ remains
an isomorphism over $X_{\mathrm{reg}}$.
We work on $M\setminus D_\infty=\pi^{-1}(X)$.
The set $\pi^{-1}(\Sing X)$, called the \emph{affine exceptional
	set}, is retained, as are its intersections with the sections.

We now express the hyperplane equations on $M$. An algebraic
line bundle is locally a product with $\C$, with linear
changes of fiber coordinate given by nowhere-vanishing regular
functions. Its sections are therefore represented locally by
regular functions with these changes of trivialization.
The hyperplane line bundle $\mathcal O_{\Pj^n}(1)$ has
homogeneous linear forms as its global sections. On the chart
$z_j\ne0$, the section defined by a linear form $\ell$ is
represented by $\ell/z_j$; see
\cite[Chapter~II, Section~7]{Hartshorne}.

Set
\[
L=\pi^*\bigl(\mathcal O_{\Pj^n}(1)|_{\overline X}\bigr),
\qquad s_0=\pi^*z_0.
\]
Here pullback means that the fiber of $L$ at $p\in M$ is the
fiber of the original line bundle at $\pi(p)$, and local
representatives of sections are composed with $\pi$.
For this algebraic construction, allow
$\xi\in\C^n\setminus\{0\}$ and put
\[
\ell(z)=\sum_{i=1}^n\xi_i z_i,\qquad s_1=\pi^*\ell.
\]
On $\pi^{-1}(\{z_j\ne0\})$, the local representatives of
$s_0$ and $s_1$ are $(z_0/z_j)\circ\pi$ and
$(\ell/z_j)\circ\pi$, respectively. They change by the same
nowhere-vanishing factor, so $s_1(p)=t s_0(p)$ defines a
global equation on $M$ for each $t\in\C$.
On $M\setminus D_\infty$, taking $j=0$ gives the representatives
$1$ and $(\ell/z_0)\circ\pi$. Thus the equation becomes
$\ell(\pi(p))=t$, where $\ell(x)=\sum_i\xi_i x_i$ in affine
coordinates.

The section $s_0$ is fixed, while $s_1$ varies in the linear
system spanned by $\pi^*z_1,\ldots,\pi^*z_n$. This system is
\emph{base-point-free} if its generating sections have no
common zero on $M$. Replacing $\xi$ by a nonzero scalar
multiple only rescales $t$. We therefore parametrize directions
by $[\xi]\in\Pj^{n-1}$ and fix a representative whenever $t$
is used. Real unit directions will be selected in
Lemma~\ref{lem:good-real-directions}.

For a morphism from a smooth $(d+1)$-dimensional variety to a
smooth curve, a reduced divisor is \emph{relative simple
	normal-crossing} if its irreducible components are smooth
and, near each of its points, there are holomorphic coordinates
$(t,u_1,\ldots,u_d)$ in which the morphism is projection to
a local coordinate $t$ on the curve and the divisor is
$u_1\cdots u_r=0$.
Thus the morphism is smooth near the divisor, but its fibers
may have singularities elsewhere.

\begin{lemma}\label{lem:infinity-model}

	There is a nonempty Zariski open set of directions
	$[\xi]\in\Pj^{n-1}$ such that
	\begin{equation}\label{eq:incidence}
		\Mcal=\{(p,t)\in M\times\C:s_1(p)=t s_0(p)\},
		\qquad f(p,t)=t,
	\end{equation}
	is smooth of dimension $d+1$, and $f:\Mcal\to\C$ is
	projective. The reduced boundary of $\Mcal$,
	\[
	\Dcal=\bigl(\Mcal\cap(D_\infty\times\C)\bigr)_{\mathrm{red}},
	\]
	is a relative simple normal-crossing divisor over all of $\C$.
	In particular, $f$ is smooth near $\Dcal$ for every finite
	complex parameter, and all intersections of components of
	$\Dcal$ are smooth over $\C$.
\end{lemma}

\begin{proof}
	The sections $\pi^*z_1,\ldots,\pi^*z_n$ have no common zero:
	such a zero would map to $[1:0:\cdots:0]$, the affine origin,
	which is excluded by \eqref{eq:origin}. Their restrictions to
	each stratum of $D_\infty$ are also base-point-free.
	Apply Bertini's theorem in characteristic zero to these
	restrictions, treating irreducible components separately when
	necessary; see \cite[Corollary~5]{Kleiman}.  For a general $s_1$, its restriction to each
	positive-dimensional stratum defines a smooth divisor or has
	no zeros. Thus, at each zero $p$ on a stratum $S$, the
	differential of a local representative satisfies
	$ds_1|_{T_pS}\ne0$.  We also choose $s_1$ nonzero at every
	zero-dimensional stratum. There are only finitely many
	strata, so these conditions hold on a common nonempty
	Zariski open set of directions. This form of Bertini requires
	no projective embedding by the chosen linear system.
	
On $M\setminus D_\infty$, the section $s_0$ does not vanish.
Thus $\Mcal\setminus\Dcal$ is the graph of the regular function
$s_1/s_0$ on $M\setminus D_\infty$, so $\Mcal$ is smooth
there. On $D_\infty\times\C$, the incidence equation
reduces to $s_1=0$; in particular, as reduced varieties,
	\[
	\Dcal=(D_\infty\cap\{s_1=0\})\times\C.
	\]
	We must also check the map in a neighborhood of this boundary.

	Let $(p,t_0)\in\Dcal$. Trivialize $L$ near $p$, using
	$s_0,s_1$ also for their local representatives, and choose
	normal-crossing coordinates for $D_\infty$, with the
	components through $p$ given by $u_1=0,\ldots,u_r=0$.
	The stratum $S$ through $p$ is locally
	$\{u_1=\cdots=u_r=0\}$. Since $s_1(p)=t_0s_0(p)=0$
	and $s_1$ was chosen nonzero on all zero-dimensional strata,
	$\dim S=d+1-r\ge1$, so $r\le d$.
	Since $ds_1(p)|_{T_pS}\ne0$, the covectors
	$du_1(p),\ldots,du_r(p),ds_1(p)$ are linearly independent.
	Keeping $u_1,\ldots,u_r$ unchanged, the holomorphic inverse
	function theorem gives coordinates $(y,u_1,\ldots,u_d)$
	with $y=s_1$. The boundary components are therefore still
	given by $u_i=0$, $1\le i\le r$.

	The zero divisor of $s_0$ is supported on $D_\infty$.
	Consequently, in this trivialization,
	\[
	s_0=g(y,u)u_1^{a_1}\cdots u_r^{a_r},\qquad s_1=y,
	\]
	where $u=(u_1,\ldots,u_d)$, $g$ is a nowhere-vanishing
	holomorphic function, and each positive integer $a_i$ is the
	vanishing order of $s_0$ along $u_i=0$.
	The incidence equation is therefore
	\begin{equation}\label{eq:local-model}
		y=t\,g(y,u)u_1^{a_1}\cdots u_r^{a_r}.
	\end{equation}
	At $(p,t_0)$, the product $u_1^{a_1}\cdots u_r^{a_r}$
	vanishes, so the derivative with respect to $y$ of the left
	side minus the right side is $1$.
	The holomorphic implicit function theorem solves uniquely
	for $y$ as a holomorphic function of $(t,u)$ near this point.
	Thus $(t,u_1,\ldots,u_d)$ are coordinates on $\Mcal$,
	$f$ is projection to $t$, and $\Dcal$ is given by
	$u_1\cdots u_r=0$. This proves the required local description at every finite
	$t_0$. For each irreducible component $E$ of $D_\infty$,
	the reduced intersection
	$(\Mcal\cap(E\times\C))_{\mathrm{red}}$ is locally given
	by a single coordinate equation $u_i=0$.
	It is therefore smooth.
	Consequently, the irreducible components of $\Dcal$ are smooth.
	Each nonempty intersection of distinct components is locally
	defined by setting some of $u_1,\ldots,u_r$ equal to zero.
	The restriction of $f$ is still projection to the free
	coordinate $t$, so every such intersection is smooth over $\C$.

Finally, \eqref{eq:incidence} defines a closed subvariety of
$M\times\C$. Since $M$ is projective, the projection
$M\times\C\to\C$ is projective. Its restriction
$f:\Mcal\to\C$ is therefore projective.
\end{proof}

Fix a direction supplied by Lemma~\ref{lem:infinity-model}. The open subset $\Mcal\setminus\Dcal$ is the graph of
$s_1/s_0$ over $M\setminus D_\infty$.
The first projection identifies it with
$M\setminus D_\infty$, so it is irreducible.
Near each point of $\Dcal$, the local coordinates above
identify $\Mcal\setminus\Dcal$ with
$\{u_1\cdots u_r\ne0\}$.
Thus every boundary point is a limit of points of this
open subset, which is therefore dense in $\Mcal$ in the
analytic topology, and hence also in the Zariski topology.
Its Zariski closure $\Mcal$ is consequently irreducible.

The linear function $\ell$ is nonconstant on $\partial K$:
otherwise it would be constant on $K$, which has nonempty
interior, contradicting $\xi\ne0$.
Since $\pi$ is surjective and $\partial K\subset X$,
we have $\ell(\partial K)\subset f(\Mcal)$, so $f$ is
nonconstant. As $\Mcal$ is irreducible, $f(\Mcal)$ is
Zariski dense in $\C$.
Since $f$ is projective, $f(\Mcal)$ is also Zariski closed,
so $f$ is surjective.

Since $\Mcal$ is smooth, generic smoothness in characteristic
zero gives a nonempty Zariski open subset of $\C$ over which
$f$ is smooth; see
\cite[Chapter~III, Corollary~10.7]{Hartshorne}.
The critical locus $\{q\in\Mcal:df_q=0\}$ is Zariski closed,
being the zero locus of the algebraic differential $df$.
Its image under the projective map $f$ is Zariski closed
and misses the nonempty open subset just obtained.
This image is therefore a proper algebraic subset of $\C$,
hence finite. Denote this set of critical values by $\Sigma$
and put
\begin{equation}\label{eq:base}
	U=\C\setminus\Sigma.
\end{equation}
In the local coordinates of Lemma~\ref{lem:infinity-model},
$f=t$ and $df=dt\ne0$ near $\Dcal$.
Thus all critical points lie outside $\Dcal$.

Write
\begin{equation}\label{eq:fibres}
	M_t=f^{-1}(t),\qquad D_t=M_t\cap\Dcal,
	\qquad Z_t=M_t\setminus D_t,\qquad
	\Zcal=\Mcal\setminus\Dcal.
\end{equation}
The first projection identifies $M_t$ with
$\{p\in M:s_1(p)=t s_0(p)\}$.
For $t\in U$, $M_t$ is a smooth projective variety of complex
dimension $d$, and $D_t$ is a simple normal-crossing divisor
in $M_t$. Thus $M_t$ compactifies $Z_t$ by adjoining $D_t$.
Under the same projection,
\[
Z_t=\{p\in M\setminus D_\infty:\ell(\pi(p))=t\}.
\]
In particular, $Z_t$ includes the points above
$\Sing X\cap\{\ell=t\}$; they are not part of $D_t$.

We next obtain local trivializations preserving the divisor.
Fix $t_*\in U$ and a small disk $B\subset U$ about $t_*$.
Since $f:f^{-1}(B)\to B$ is a proper surjective smooth
submersion, Ehresmann's theorem gives smooth local
trivializations \cite[Lemma~9.2]{KolarMichorSlovak}.
To preserve $\Dcal$, use the vector-field proof of
Ehresmann's theorem with the following additional choice.
A local lift of a smooth real vector field $V$ on $B$
is a smooth real vector field $\widetilde V$ on an open
subset of $f^{-1}(B)$ satisfying
$df(\widetilde V)=V\circ f$.
In relative normal-crossing coordinates, lift the two real
coordinate vector fields on $B$ by keeping the $u$-coordinates
fixed. These lifts are tangent to every boundary component.
Away from $\Dcal$, choose smooth lifts using the submersion
property. A smooth partition of unity on $f^{-1}(B)$ subordinate
to these neighborhoods combines the local lifts.
Since its coefficients sum to one, the resulting vector
fields still lift the corresponding coordinate vector
fields on $B$ and remain tangent to every boundary component.

Use these lifts to transport points along the straight
segments from $t_*$ to nearby points of $B$. 
After shrinking $B$ so that its closure lies in the original
disk, properness makes $f^{-1}(\overline B)$ compact.
Each lifted path stays in this compact set, so the
continuation theorem for ordinary differential equations
ensures existence along the whole segment, in either
direction.
Smooth dependence on the endpoint $t$ and the initial point
gives a smooth family of diffeomorphisms $M_{t_*}\to M_t$,
with inverses obtained by reversing the paths.
Tangency to each component of $\Dcal$ ensures that these
diffeomorphisms map $D_{t_*}$ onto $D_t$.
We obtain a smooth local trivialization of pairs
\[
\bigl(f^{-1}(B),\Dcal\cap f^{-1}(B)\bigr)
\simeq
\bigl(B\times M_{t_*},B\times D_{t_*}\bigr)
\]
compatible with projection to $B$.

These trivializations also identify the complements $Z_t$.
Consequently, $H_d(Z_t;\Q)$ and $H_d(M_t;\Q)$ form
\emph{local systems} on $U$, meaning locally constant
families of vector spaces. The induced isomorphisms on homology along paths depend
only on the homotopy class of the path in $U$ with fixed
endpoints and are called \emph{Gauss--Manin transport}.
Since the same trivializations are used for the compact
fibers and their complements, the maps
$H_d(Z_t;\Q)\to H_d(M_t;\Q)$ induced by the inclusions
$Z_t\hookrightarrow M_t$ commute with this transport.
The argument uses properness of $f$, not of the restriction
$\Zcal\cap f^{-1}(U)\to U$, which need not be proper.

\subsection{A single ordinary tangency}

We now choose real directions $\xi$ for which the fiber
$M_{b_K(\xi)}$, where $b_K(\xi)$ is defined in
\eqref{eq:support-interval}, contains exactly one critical point
of $f$, and this point is nondegenerate.
A nondegenerate real support point alone does not exclude other
complex critical points in the same fiber. We first control
contacts on the smooth locus of $\overline X$, then exclude
critical points over $\Sing X$. Together with
Lemma~\ref{lem:infinity-model}, this gives the single critical
point required in Lemma~\ref{lem:compatible-pl}.

The dual projective space $(\Pj^n)^*$ parametrizes projective
hyperplanes in $\Pj^n$. For an irreducible projective variety
$Y\subsetneq\Pj^n$ and $p\in Y_{\mathrm{reg}}$, let
$\mathbb T_pY\subset\Pj^n$ denote the embedded projective
tangent space, to distinguish it from the vector tangent
space $T_pY$. The projective dual
$Y^*\subset(\Pj^n)^*$ is the Zariski closure of the set of
hyperplanes containing $\mathbb T_pY$ for some
$p\in Y_{\mathrm{reg}}$. If $Y$ is a hypersurface,
$\mathbb T_pY$ is itself a hyperplane, and the projective
Gauss map is $p\mapsto\mathbb T_pY$ on $Y_{\mathrm{reg}}$.

We use the reflexivity theorem for projective duality over
$\C$, including biduality $(Y^*)^*=Y$; see
\cite[Section~1.3]{Tevelev}. More precisely, for
$p\in Y_{\mathrm{reg}}$ and $H\in(Y^*)_{\mathrm{reg}}$,
the hyperplane $H$ is tangent to $Y$ at $p$ if and only if
the hyperplane in $(\Pj^n)^*$ parametrizing hyperplanes
through $p$ is tangent to $Y^*$ at $H$.

Let $Y$ be a hypersurface, and let an affine hyperplane
$H=\{\ell=t_0\}$ pass through
$p\in Y_{\mathrm{reg}}\cap\C^n$. We say that $H$ has an
\emph{ordinary quadratic tangency} with $Y$ at $p$ if
$\ell|_Y$ has a nondegenerate critical point there:
its differential vanishes, and its Hessian in local
holomorphic coordinates on $Y$ is nonsingular.

\begin{lemma}\label{lem:dual}
	Let $X\subset\C^n$ be the complex Zariski closure of the
	$C^\infty$ semialgebraic boundary of a convex body
	$K\subset\R^n$, as in Lemma~\ref{lem:nash-boundary}.
	The projective dual $\overline X^*\subset(\Pj^n)^*$ of its
	projective closure is an irreducible hypersurface.
	Outside a proper algebraic subset of $\overline X^*$, a
	hyperplane has exactly one point of tangency on the smooth
	locus of $\overline X$. This point lies in $X_{\mathrm{reg}}$,
	and the contact is an ordinary quadratic tangency.
\end{lemma}

\begin{proof}
	Let $\nu_K:\partial K\to\Sph$ be the outward Gauss map.
	It 
	is a Nash map because $\partial K$ is a Nash hypersurface.
	A value is regular if $d\nu_K$ is surjective at every point
	of its inverse image. The critical values form a compact set,
	since $\partial K$ is compact, and have measure zero by Sard's
	theorem \cite[Theorem~6.10]{Lee}. Thus the regular values form
	an open dense subset of $\Sph$. With the convention
	$\mathrm{II}_p(v,w)=d\nu_K|_p(v)\cdot w$, convexity makes
	$\mathrm{II}_p$ positive semidefinite. If $\nu_K(p)$ is a regular
	value, then $d\nu_K|_p$ is invertible, so $\mathrm{II}_p$
	is positive definite.
	
	We also need a support point outside $\Sing X$.
	Since $\partial K$ is Zariski dense in $X$ and $\Sing X$ is a
	proper algebraic subset defined over $\R$, there is a real
	polynomial vanishing on $\Sing X$ but not identically on
	$\partial K$. Its restriction is a nonzero Nash function on
	the connected Nash manifold $\partial K$.
	By \cite[Proposition~8.1.10 and Theorem~2.8.8]{BCR}, its zero
	set has dimension at most $n-2$, and the Gauss map does not
	increase semialgebraic dimension. Hence
	$\nu_K(\partial K\cap\Sing X)$ is compact and semialgebraic
	of real dimension at most $n-2$.
Choose a regular value $\xi$ outside this image and a support
point $p$ with $\nu_K(p)=\xi$. Then $p\in X_{\mathrm{reg}}$
and $d\nu_K|_p$ is invertible.

Let $\gamma$ denote the projective Gauss map of $\overline X$.
For $q\in\partial K$ near $p$, we have
\[
\gamma(q)=[-q\cdot\nu_K(q):\nu_K(q)].
\]
Projecting to the last $n$ homogeneous coordinates gives
$[\nu_K(q)]$. Since the projection
$\Sph\to\mathbb{RP}^{n-1}$ is a local diffeomorphism,
the restriction of $\gamma$ to $\partial K$ has real rank
$n-1$ at $p$.
In local holomorphic coordinates that are real on the real
locus, the Jacobian of $\gamma$ at $p$ is a real matrix.
Its rank over $\C$ is therefore also $n-1$.

Since $\dim_{\C}\overline X=n-1$, the image closure
$\overline X^*$ has complex dimension $n-1$.
It is irreducible as the image closure of the irreducible
smooth locus of $\overline X$, and is therefore a hypersurface.
	
	To prove that a general tangent hyperplane has a unique
	contact point, let $\mathcal C_X$ be the Zariski closure
	of the graph
	\[
	\{(p,\mathbb T_p\overline X):p\in X_{\mathrm{reg}}\}
	\subset\Pj^n\times(\Pj^n)^*.
	\]
	Since $X_{\mathrm{reg}}$ is dense in
	$(\overline X)_{\mathrm{reg}}$, the same closure is obtained
	using all smooth points of $\overline X$.
	It is irreducible of dimension $n-1$, and the second projection
	$q:\mathcal C_X\to\overline X^*$ is proper and surjective.
	By reflexivity \cite[Section~1.3]{Tevelev}, exchanging the two
	factors identifies $\mathcal C_X$ with the closure of the
	graph of the Gauss map of the hypersurface $\overline X^*$.
	This Gauss map is regular on $(\overline X^*)_{\mathrm{reg}}$,
	so its graph closure agrees with its graph over that locus.
	Hence $q$ is an isomorphism over
	$(\overline X^*)_{\mathrm{reg}}$, and in particular is birational.
	Geometrically, for a general $H$ with contact point $p$,
	the tangent hyperplane $\mathbb T_H(\overline X^*)$
	parametrizes all hyperplanes in $\Pj^n$ through $p$,
	and thus determines $p$ uniquely.

	The displayed graph is a dense Zariski open subset of
	$\mathcal C_X$. Since $q$ is birational, choose a dense open
	subset $\mathcal C_X^\circ\subset\mathcal C_X$ contained in
	this graph such that $q$ restricts to an isomorphism onto
	a smooth open subset of $\overline X^*$.
	Its complement is closed of dimension at most $n-2$.
	By properness of $q$, the set
	$B_0=q(\mathcal C_X\setminus\mathcal C_X^\circ)$
	is closed algebraic, and its dimension is at most $n-2$.
	For $H\in\overline X^*\setminus B_0$, surjectivity of $q$
	ensures that $q^{-1}(H)$ is nonempty. This fiber lies entirely
	in $\mathcal C_X^\circ$, so it consists of one point $(p,H)$,
	with $p\in X_{\mathrm{reg}}$.
	The first projection identifies the displayed graph with
	$X_{\mathrm{reg}}$. Hence the projective Gauss map has rank
	$n-1$ at $p$.
	
	Choose affine coordinates $(u,x_n)\in\C^{n-1}\times\C$
	in which $p=0$ and $H=\{x_n=0\}$.
	Then $X$ is locally the graph $x_n=\phi(u)$, where $\phi$
	is holomorphic, $\phi(0)=0$, and $D\phi(0)=0$.
	In dual affine coordinates $(a,b)$ representing the
	hyperplane $x_n=a\cdot u+b$, the Gauss map is
	\[
	u\longmapsto\bigl(D\phi(u),\,\phi(u)-u\cdot D\phi(u)\bigr).
	\]
	Its differential at $0$ sends $v$ to
	$\bigl(D^2\phi(0)v,0\bigr)$.
	Since the Gauss map has rank $n-1$ there, $D^2\phi(0)$
	is nonsingular. The defining function $x_n$ of $H$
	restricts to $\phi$ on $X$, so the tangency is ordinary
	quadratic.

\end{proof}

\begin{lemma}\label{lem:good-real-directions}
	There is a set $\mathcal G\subset\Sph$, dense in the usual
	topology, such that, for every $\xi\in\mathcal G$, the
	conclusions of Lemma~\ref{lem:infinity-model} hold for
	$[\xi]$ and the fiber $M_{b_K(\xi)}$ contains exactly one
	critical point of $f$.
	Under the map $(p,t)\mapsto\pi(p)$, this point maps to the
	unique support point of $K$ in direction $\xi$.
	The support point lies in $X_{\mathrm{reg}}$, and the
	critical point of $f$ is nondegenerate.
\end{lemma}

\begin{proof}
	Let $B_0\subset\overline X^*$ be the exceptional subset
	constructed in the proof of Lemma~\ref{lem:dual}; its complex
	dimension is at most $n-2$.
	Every $H\in\overline X^*\setminus B_0$ has exactly one
	contact point on $(\overline X)_{\mathrm{reg}}$.
	This point lies in $X_{\mathrm{reg}}$, and the tangency is
	ordinary quadratic.
	We must also exclude critical points over the affine
	singular locus $\Sing X$.

Decompose the affine exceptional set
$\pi^{-1}(\Sing X)\subset M\setminus D_\infty$ into finitely
many smooth irreducible locally closed complex algebraic
pieces $T$ such that $\pi|_T$ is a submersion onto a smooth
irreducible locally closed subset $S\subset\Sing X$.
To obtain this decomposition, let $Z$ be an irreducible
component of the remaining set, and let $Y$ be the Zariski
closure of $\pi(Z)$ in $\Sing X$.
Choose a smooth dense open subset of $Z$, disjoint from
the other components and mapping into $Y_{\mathrm{reg}}$.
By generic smoothness
\cite[Chapter~III, Lemma~10.5]{Hartshorne}, it contains a
dense open subset $T$ such that
$\pi|_T:T\to Y_{\mathrm{reg}}$ is smooth.
Since smooth morphisms are open, $S=\pi(T)$ is open in
$Y_{\mathrm{reg}}$; it is smooth and irreducible, and
$\pi|_T:T\to S$ is a surjective submersion.
Doing this for each component leaves a closed subset of
smaller dimension. Induction on dimension gives the required
finite decomposition.
	
On $\Mcal\setminus\Dcal$, the first projection identifies
$f$ with $\ell\circ\pi$ on $M\setminus D_\infty$.
Let $\widetilde p\in T$ be a critical point of
$\ell\circ\pi$, and set $p=\pi(\widetilde p)\in S$.
Restricting its differential to $T_{\widetilde p}T$ gives
\[
d\ell_p\circ d\pi_{\widetilde p}\big|_{T_{\widetilde p}T}=0.
\]
Since $d\pi_{\widetilde p}(T_{\widetilde p}T)=T_pS$, we have
$d\ell_p|_{T_pS}=0$.
Thus the hyperplane $\{\ell=\ell(p)\}$ contains the affine
tangent space $p+T_pS$.
Since $S$ is locally closed in $\Pj^n$, it is open in its
projective closure $\overline S$. As $S$ is smooth,
$p\in(\overline S)_{\mathrm{reg}}$.
The projective closure of this hyperplane therefore contains
$\mathbb T_p\overline S$ and belongs to $\overline S^{\,*}$.

	Since $S\subset\Sing X$, its projective closure has dimension
	at most $n-2$ and is a proper subvariety of $\overline X$.
	We claim that, for every such $S$,
	\begin{equation}\label{eq:dual-exclusion}
		\overline X^*\not\subset\overline S^{\,*}.
	\end{equation}
	To see this, put $k=\dim\overline S$.
	For $p\in(\overline S)_{\mathrm{reg}}$, the hyperplanes
	$H$ containing $\mathbb T_p\overline S$ form a projective
	space of dimension $n-k-1$.
	The pairs $(p,H)$ form a projective-space bundle over
	$(\overline S)_{\mathrm{reg}}$, so their total space is
	irreducible of dimension $k+(n-k-1)=n-1$.
	Its image closure $\overline S^{\,*}$ is therefore irreducible
	of dimension at most $n-1$.
	If it contained the irreducible hypersurface $\overline X^*$,
	they would be equal. Biduality would then give
	$\overline S=\overline X$, contradicting
	$\overline S\subsetneq\overline X$.
	This proves \eqref{eq:dual-exclusion}.
	
	Set
	\[
	B_1=B_0\cup\bigcup_S
	\bigl(\overline X^*\cap\overline S^{\,*}\bigr),
	\]
	where the union runs over the finitely many target sets $S$
	above. Each added intersection is a proper algebraic subset
	of the irreducible hypersurface $\overline X^*$.
	Thus $B_1$ is algebraic of dimension at most $n-2$.
	
	We now pass from hyperplanes to directions. The map
	\[
	\rho:(\Pj^n)^*\setminus\{H_\infty\}\longrightarrow\Pj^{n-1},
	\qquad
	\rho\!\left(\left\{\sum_{i=1}^n\eta_i z_i=s z_0\right\}\right)
	=[\eta]
	\]
	is a morphism; its domain consists of the projective closures
	of affine hyperplanes. The Zariski closure of
	$\rho(B_1\setminus\{H_\infty\})$ has dimension at most $n-2$,
	so it is a proper algebraic subset of $\Pj^{n-1}$.
For $[\eta]$ outside this closure, the projective closure
of $\{\eta\cdot x=s\}$ lies outside $B_1$ for every $s\in\C$.
Intersect this nonempty Zariski open set of directions with
the one in Lemma~\ref{lem:infinity-model}.
Since $\Pj^{n-1}$ is irreducible, the resulting Zariski open
set $\Omega$ is nonempty.
	
	The real unit vectors with $[\xi]\in\Omega$ form an open
	dense subset of $\Sph$ in the usual topology.
	Indeed, a proper complex projective algebraic subset cannot
	contain a nonempty open set of real directions: a nonzero
	homogeneous polynomial vanishing there would vanish on an
	open cone in $\R^n$, and hence be the zero polynomial.
Let $\mathcal G$ consist of those $\xi\in\Sph$ such that
$[\xi]\in\Omega$, $\xi$ is a regular value of $\nu_K$, and
$\xi\notin\nu_K(\partial K\cap\Sing X)$.
As shown in the proof of Lemma~\ref{lem:dual}, the critical
values of $\nu_K$ and the set
$\nu_K(\partial K\cap\Sing X)$ are compact with empty interior.
Their complements are open and dense, so $\mathcal G$ is
open and dense in $\Sph$ in the usual topology.
	
	Fix $\xi\in\mathcal G$, and write $\ell(x)=\xi\cdot x$.
	The inverse function theorem makes $\nu_K^{-1}(\xi)$ discrete,
	since $\xi$ is a regular value. This inverse image is exactly
	the nonempty convex support face
	$K\cap\{x:\xi\cdot x=b_K(\xi)\}$.
	A discrete convex set consists of one point; denote it by
	$p_\xi$. By the choice of $\mathcal G$, it lies in
	$X_{\mathrm{reg}}$.
	
Let $H$ be the projective closure of
$\{x:\xi\cdot x=b_K(\xi)\}$.
Since $p_\xi\in X_{\mathrm{reg}}$, the tangent space
$T_{p_\xi}X$ is the complexification of
$T_{p_\xi}(\partial K)$.
Thus $H$ is tangent to $X$ at $p_\xi$, so
$H\in\overline X^*$, and our choice of direction gives
$H\notin B_1$.
	Lemma~\ref{lem:dual} therefore makes $p_\xi$ its only contact
	point on the smooth locus of $\overline X$, with ordinary
	quadratic contact. Over $X_{\mathrm{reg}}$, a point of this
	fiber is critical for $f$ exactly when $H$ is tangent to $X$
	at its image, because $\pi$ is an isomorphism there.
	Hence the lift of $p_\xi$ is the only critical point over
	the affine smooth locus.
Any critical point of $f$ in this fiber lying over $\Sing X$
would force $H\in\overline S^{\,*}$ for one of the sets $S$
above, and hence $H\in B_1$, a contradiction.
	Lemma~\ref{lem:infinity-model} excludes critical points on
	$\Dcal$. These three cases cover the entire fiber
	$M_{b_K(\xi)}$.
	
	Finally, near the lift of $p_\xi$, the first projection and
	$\pi$ identify $f$ with $\ell|_X$. Its Hessian is therefore
	nonsingular by the ordinary quadratic contact just proved.
	This gives the asserted nondegeneracy.
\end{proof}

\begin{remark}\label{rem:affine-exceptional}
	The preceding argument excludes critical points over
	$\Sing X$ at the chosen support value; it does not remove
	the affine exceptional set. The points lying over
	$\Sing X\cap\{\ell=t\}$ remain in $Z_t$, not in the boundary
	divisor $D_t$.
	Since the composite map
	$\Zcal\to M\setminus D_\infty\to X\hookrightarrow\C^n$
	is regular, the ambient polynomial differential form used
	below pulls back to a holomorphic form on $\Zcal$,
	including at these points.
	No smoothness of $\Sing X$ is assumed; only the finite
	algebraic decomposition used above is needed.
\end{remark}

\section{Vanishing periods with finite branching}\label{sec:semisimplicity}

Let $f:\Mcal\to\C$ and $\Dcal\subset\Mcal$ be as above, with smooth projective fibers over $U=\C\setminus\Sigma$. Fix $\tau\in U$ and put
\begin{equation}\label{eq:homology}
 \mathsf A=H_d(Z_\tau;\Q),\qquad
 \mathsf H=H_d(M_\tau;\Q),\qquad
 j:\mathsf A\longrightarrow\mathsf H.
\end{equation}
Here $j$ is induced by the inclusion $Z_\tau\hookrightarrow M_\tau$.
A cycle representing a nonzero class in $\mathsf A$ may bound
a chain in $M_\tau$, so $j$ need not be injective.

The complex structure orients the compact real $2d$-manifold
$M_\tau$, which has no boundary. Its middle-dimensional
intersection pairing on $\mathsf H$, denoted by
$\langle\ ,\ \rangle$, is computed by counting signed
intersections of transverse representatives.
Poincar\'e duality makes this pairing nondegenerate.
Interchanging the two classes gives the sign $(-1)^d$,
so the pairing is alternating when $d$ is odd; see
\cite[Section~3.3]{Hatcher}.
If $M_\tau$ is disconnected, the pairing is the direct sum
of the pairings on its connected components.
The argument below uses the full homology group $\mathsf H$.

A nondegenerate, or Morse, critical point of a holomorphic function is a point where its differential vanishes and its complex Hessian is nonsingular. The holomorphic Morse lemma gives local coordinates $u_1,\ldots,u_{d+1}$ in which the function minus its critical value is $u_1^2+\cdots+u_{d+1}^2$. A nearby fiber contains a real $d$-sphere that contracts to the critical point. The homology class of this sphere in the nearby fiber is called
a \emph{vanishing cycle}. A \emph{Milnor ball} is a sufficiently
small closed ball centered at the critical point in the above
local coordinates, chosen so that its boundary is transverse
to every sufficiently nearby regular fiber. We use these constructions and the Picard--Lefschetz formula in the form developed in \cite{AGV,Milnor,Vassiliev2002}.

Suppose that the fiber over $t_0\in\Sigma$ has exactly one critical point of $f$, and that this point is nondegenerate and lies away from $\Dcal$. Choose a path in $U$ from $\tau$ to a regular value sufficiently
close to $t_0$. Transporting the local vanishing sphere back
to $\tau$ along the reverse path defines a class
$\delta\in\mathsf A$. We also write
\[
 v=j\delta\in\mathsf H.
\]

\begin{lemma}\label{lem:compatible-pl}
A based meridian of $t_0$ follows the chosen path from $\tau$, traverses one small positively oriented circle around $t_0$, and returns along the path. We use $T$ for the monodromy automorphisms induced by this
loop on both $\mathsf A$ and $\mathsf H$. For one fixed sign $\epsilon\in\{1,-1\}$, the Picard--Lefschetz transformations are
\begin{align}
 Ta&=a+\epsilon\langle ja,v\rangle\delta,
 &&a\in\mathsf A,\label{eq:pl-open}\\
 Th&=h+\epsilon\langle h,v\rangle v,
 &&h\in\mathsf H.\label{eq:pl-compact}
\end{align}
\end{lemma}

\begin{proof}
Choose a Milnor ball around the unique critical point, disjoint from the boundary divisor. On the complement of its interior, $f$ is a submersion over a sufficiently small disk about $t_0$, the divisor remains relative simple normal crossing, and the boundary sphere of the Milnor ball is transverse to the fibers. The boundary-preserving lifting argument following \eqref{eq:fibres} applies here as well, with lifts also tangent to this boundary sphere, which is disjoint from the divisor. Properness gives a simultaneous trivialization of the complement. The local monodromy can consequently be represented by a map supported in the Milnor ball. Transport back to the reference fiber gives the based
monodromy used in the statement.

The local Picard--Lefschetz formula adds the vanishing sphere multiplied by its intersection number with the transported cycle. For a compact cycle in the open fiber, this local intersection number is the intersection number of its image in $M_\tau$ with $v$. The same local ball and orientations give \eqref{eq:pl-open} and \eqref{eq:pl-compact}; the choice of convention affects only $\epsilon$. This is the usual Picard--Lefschetz formula and its functoriality under inclusion; compare \cite[Section~2, proof of Theorem~1]{Vassiliev2018}. The trivialization outside the ball preserves the divisor, so there is no additional term supported at infinity.
\end{proof}

Let $\lambda$ be a single-valued holomorphic $d$-form on $\Zcal$, whose restriction $\lambda_t$ to each smooth fiber is closed. For $a\in\mathsf A$, write
\[
 I_a(t)=\int_{a(t)}\lambda_t
\]
for its local period, where $a(t)$ is Gauss--Manin transport near $\tau$. The period is holomorphic. Here $f^*(dt)=df$ is the pullback of the coordinate one-form $dt$, and $f^*(dt)\wedge\lambda$ is closed because it is a top-degree holomorphic form on the smooth total space. Compact representatives may be transported over a small disk; applying Stokes' theorem to this form and then Morera's theorem proves holomorphy. We use the same argument across critical values in Proposition~\ref{prop:boundary-entire}.

A function germ at $\tau$ records a function only up to agreement on a neighborhood of $\tau$. For a loop $g$ in $U$ based at $\tau$, let $ga$ denote the
class obtained by transporting $a$ around $g$. Analytic
continuation of $I_a$ along this loop gives the germ $I_{ga}$. A period has \emph{finite branching} on $U$ if it has only finitely many germs at $\tau$ under all such continuations. This condition refers to the period, not to the orbit of the homology class.

\begin{proposition}\label{prop:semisimplicity}
Under the preceding hypotheses, assume that $d$ is odd. If $I_\delta$ is not identically zero and has finite branching on $U$, then
\begin{equation}\label{eq:null-compact}
 j\delta=0.
\end{equation}
\end{proposition}

\begin{proof}
Let $G=\pi_1(U,\tau)$ be the fundamental group of based loops
modulo homotopy, acting by monodromy. Since $\langle v,v\rangle=0$,
formula \eqref{eq:pl-open} gives, for every $g\in G$ and
$k\in\mathbb Z$,
\[
 T^k(g\delta)=g\delta+
 k\epsilon\langle j(g\delta),v\rangle\delta.
\]
Taking periods gives an identity of holomorphic germs at $\tau$:
\begin{equation}\label{eq:finite-values}
 I_{T^kg\delta}=I_{g\delta}+
 k\epsilon\langle j(g\delta),v\rangle I_\delta.
\end{equation}
For each fixed $g$, finite branching allows only finitely many
germs on the left as $k$ varies. Since $I_\delta$ is a nonzero
germ, the coefficient of $k$ must vanish. Thus
\begin{equation}\label{eq:orbit-orthogonal}
 \langle j(g\delta),v\rangle=0\qquad(g\in G).
\end{equation}

Set
\[
 W=\Span_{\Q}\{j(g\delta):g\in G\}\subset\mathsf H.
\]
Inclusion commutes with transport, as explained after \eqref{eq:fibres}, so $j(ga)=g(ja)$. Hence $W$ is a $G$-invariant subspace. By \eqref{eq:pl-compact} and \eqref{eq:orbit-orthogonal}, $T$ acts as the identity on $W$.

A representation is semisimple if each invariant subspace has an invariant direct-sum complement. By Lemma~\ref{lem:infinity-model} and \eqref{eq:base}, the restriction $f^{-1}(U)\to U$ is a smooth projective morphism over the algebraic curve $U$. Deligne's semisimplicity theorem therefore applies to its cohomology local system with fibers $H^d(M_t;\Q)$; see \cite[Proposition~4.2.5(i), (iii) and Theorem~4.2.6]{Deligne1971}. Dualizing gives semisimplicity of the representation on $\mathsf H$. Thus there is a $G$-invariant complement
\begin{equation}\label{eq:semisimple-splitting}
 \mathsf H=W\oplus W'.
\end{equation}
The complement $W'$ is invariant under all monodromy transformations, not only under $T$. Let $\mathrm{Id}$ denote the identity on $\mathsf H$. By \eqref{eq:pl-compact}, $(T-\mathrm{Id})\mathsf H\subset\Q v\subset W$. Since $W'$ is $T$-invariant,
\[
 (T-\mathrm{Id})W'\subset W\cap W'=0.
\]
Therefore $T$ is the identity on all of $\mathsf H$. If $v\ne0$, nondegeneracy of the compact intersection pairing gives $h\in\mathsf H$ with $\langle h,v\rangle\ne0$, contradicting \eqref{eq:pl-compact}. Hence $v=0$.
\end{proof}

\begin{remark}\label{rem:global-semisimplicity}
Only the compact local system $H_d(M_t;\Q)$ is asserted to be semisimple. The open local system $H_d(Z_t;\Q)$ need not be semisimple. Semisimplicity concerns the whole monodromy representation; it does not assert that each individual monodromy operator is diagonalizable. Also, the base for Deligne's theorem is the algebraic curve $U$, not a small analytic punctured disk. The passage from finite branching to \eqref{eq:null-compact} uses \eqref{eq:finite-values}; it does not assume that a nonzero homology class has a nonzero period for every differential form.
\end{remark}

\begin{remark}\label{rem:ball-cycle}
For the unit ball in $\R^3$ and an internal section at height $t$, the complexification of the section boundary is a smooth affine conic, isomorphic to $\C^*=\C\setminus\{0\}$. Its smooth projective compactification is $\Pj^1$. The real section circle is nonzero in $H_1(\C^*;\Q)$ but maps to zero in $H_1(\Pj^1;\Q)$, while the section area is the nonzero polynomial $\pi(1-t^2)$. Thus $j\delta=0$ does not imply that the affine period vanishes.
\end{remark}

\section{Periods supported near infinity}\label{sec:boundary-periods}

We now use the boundary regularity over the full parameter line provided by Lemma~\ref{lem:infinity-model}. A class that becomes zero in compact homology can be represented near the deleted divisor. The advantage is that this neighborhood remains smooth even when the interior fiber is singular. We prove holomorphic extension of its period, thereby excluding finite poles as well as branching.

\begin{proposition}\label{prop:boundary-entire}
Let $f:\Mcal\to\C$ be a projective morphism from a smooth complex algebraic variety of dimension $d+1$. Let $\Dcal\subset\Mcal$ be a reduced divisor such that $f$ is smooth near $\Dcal$ and $\Dcal$ is a relative simple normal-crossing divisor over all of $\C$. Suppose $f$ is smooth over $U=\C\setminus\Sigma$, where $\Sigma$ is finite. Fix $\tau\in U$ and use the notation \eqref{eq:fibres} and \eqref{eq:homology}. Let $\lambda$ be a holomorphic $d$-form on $\Zcal$, with closed restrictions to smooth fibers.

If $a\in H_d(Z_\tau;\Q)$ satisfies $ja=0$, then its period germ $I_a$ extends to a single-valued entire function of $t$.
\end{proposition}

\begin{proof}
If $D_\tau$ is empty, then $j$ is the identity and $a=0$, so there is nothing to prove. Assume $D_\tau\ne\varnothing$.

We work over disks centered at $\tau$, allowing their radii to increase. Fix such a disk $B$ and work initially over a slightly larger disk. All compactness statements below are over $\overline B$.

\smallskip
\noindent\emph{A compact boundary family.}
We shall construct compact real hypersurfaces near $D_t$, but disjoint from it, which vary smoothly over $B$. These will carry the integration cycles; no extension of $\lambda$ across $\Dcal$ is required.
Local reduced defining equations of $\Dcal$ differ on overlaps
by nowhere-vanishing holomorphic factors. They therefore define
a line bundle $\mathcal O(\Dcal)$ and a section $\sigma$ locally
represented by these equations. Choose a smoothly varying
Hermitian norm on its one-dimensional fibers and set
$\rho=\|\sigma\|^2$. This gives a globally defined smooth
nonnegative function whose zero set is exactly $\Dcal$. In relative normal-crossing coordinates,
\begin{equation}\label{eq:rho-local}
 \rho=\vartheta(t,z)\prod_{i=1}^r|z_i|^2,
 \qquad \vartheta(t,z)>0.
\end{equation}
Here $\vartheta$ is smooth. A tangent vector is called vertical when it is tangent to a fiber, that is, when $t$ is fixed. Thus the coordinates other than $t$ are fiber coordinates. For a complex coordinate $z_i=x_i+\mathrm{i}y_i$, we use
$\partial/\partial z_i
=\frac12(\partial/\partial x_i-\mathrm{i}\partial/\partial y_i)$.  On the complement of the divisor,
\[
 \frac{\partial\log\rho}{\partial z_i}
 =\frac1{z_i}+\frac{\partial\log\vartheta}{\partial z_i},
 \qquad 1\le i\le r.
\]
On a relatively compact coordinate chart, the derivatives of $\log\vartheta$ are bounded, say by $C$. After shrinking about a point of the divisor, one of the vanishing coordinates satisfies $|z_i|<(2C+1)^{-1}$. At points off the divisor,
\[
 \left|\frac{\partial\log\rho}{\partial z_i}\right|
 \ge |z_i|^{-1}-C>0.
\]
Since $z_i$ is a fiber coordinate, $\rho$ has no vertical critical point there. Cover the compact set $\Dcal\cap f^{-1}(\overline B)$ by finitely many such neighborhoods on which $f$ is smooth. The complement of their union in $f^{-1}(\overline B)$ is compact and disjoint from $\Dcal$, so $\rho$ has a positive minimum there if the complement is nonempty. Choosing $\varepsilon>0$ sufficiently small, we obtain that $f$ is smooth on $\{\rho\le2\varepsilon\}$ over $\overline B$, and that $\rho$ has no vertical critical point where $0<\rho\le2\varepsilon$.

Consequently,
\[
 E_B=\{x\in f^{-1}(B):\rho(x)=\varepsilon\}
 \longrightarrow B
\]
is a proper smooth submersion. Its fibers $E_t$ are compact. To see surjectivity of its differential, lift a base tangent vector using the smooth map $f$ and correct the lift by a vertical vector on which $d\rho\ne0$.
After decreasing $\varepsilon$ if necessary, $E_\tau$ is nonempty:
by \eqref{eq:rho-local}, a fiber neighborhood of a point of $D_\tau$
contains points at every sufficiently small positive level of $\rho$. The image of the map $E_B\to B$ is nonempty, open and closed
in the connected disk $B$, so this map is surjective. Ehresmann's theorem \cite[Lemma~9.2]{KolarMichorSlovak} makes $E_B\to B$ a smooth fiber bundle. Since $B$ is a disk, its homology local system is trivial.

For
\[
 N_t=\{x\in M_t:\rho(x)<2\varepsilon\},
\]
the punctured neighborhood $N_t\setminus D_t$ retracts onto $E_t$. Choose a smooth Riemannian metric in the fiber directions,
and let $\nabla_{\mathrm{vert}}\rho$ denote the gradient of
$\rho|_{M_t}$. On $0<\rho<2\varepsilon$ over $B$, define
\[
 Y=\frac{\nabla_{\mathrm{vert}}\rho}
 {|\nabla_{\mathrm{vert}}\rho|^2},
\]
where the norm is taken with respect to the same metric.
If $\Phi_s$ denotes its local flow, then
$\rho(\Phi_s(x))=\rho(x)+s$. The required deformation retraction is
\[
H_u(x)=\Phi_{u(\varepsilon-\rho(x))}(x),
\qquad 0\le u\le1.
\] Each such trajectory stays in a compact strip separated from $D_t$, so no flow to level zero is required. This construction is compatible with the parameter.

\smallskip
\noindent\emph{Lifting the cycle to the boundary neighborhood.}
At the regular base point, apply the Mayer--Vietoris sequence for the open cover $M_\tau=Z_\tau\cup N_\tau$; see \cite[Section~2.2, pp.~149--150]{Hatcher}:
\begin{equation}\label{eq:mv}
 H_d(N_\tau\setminus D_\tau;\Q)
 \longrightarrow H_d(Z_\tau;\Q)\oplus H_d(N_\tau;\Q)
 \longrightarrow H_d(M_\tau;\Q).
\end{equation}
The first arrow is induced by the two inclusions of the
intersection, with opposite signs. The second adds the images
of the two classes in $H_d(M_\tau;\Q)$. Since $ja=0$, the pair
$(a,0)$ lies in the kernel of the second arrow. Exactness means
that this kernel is the image of the first arrow, so there is
a class
\[
 b\in H_d(N_\tau\setminus D_\tau;\Q)
 \simeq H_d(E_\tau;\Q)
\]
whose image in $H_d(Z_\tau;\Q)$ is $a$. This argument does not require $H_d(N_\tau;\Q)=0$.

Transport $b$ in the compact boundary family and denote the resulting class in $E_t$ by $b_t$. Define
\begin{equation}\label{eq:boundary-period}
 J_B(t)=\int_{b_t}\lambda_t.
\end{equation}
At a critical value, $\lambda_t$ in \eqref{eq:boundary-period} means the restriction of $\lambda$ to the smooth part of the fiber containing $E_t$. It is a closed top-degree holomorphic form there, so the integral depends only on the homology class $b_t$. Representatives can be chosen as finite smooth singular cycles varying locally smoothly in $t$. On a compact subdisk they lie in a common compact set disjoint from $\Dcal$. Thus $J_B$ is continuous, including at values where the interior of $M_t$ is singular. Near $\tau$, work in a disk contained in $U$.
The inclusions $E_t\hookrightarrow Z_t$ induce a map of homology
local systems: in local trivializations of the two families,
they form a continuous family of maps and hence induce the
same homology map. Thus $b_t$ maps to the Gauss--Manin
transport of $a$, and $J_B=I_a$ as germs.

\smallskip
\noindent\emph{Holomorphy across every finite parameter.}
The neighborhood containing $E_B$ is a complex manifold of dimension $d+1$ on which $f$ is a holomorphic submersion. There,
\[
 \omega=f^*(dt)\wedge\lambda
\]
is a holomorphic top-degree form and hence $d\omega=0$. Let $\Delta\subset B$ be a sufficiently small closed triangle. Choose $t_*\in\Delta$ and a finite piecewise smooth cycle
in $E_{t_*}$ representing $b_{t_*}$.
Trivialize the smooth boundary bundle over $\Delta$ and sweep
this cycle over the triangle. This gives a real $(d+2)$-chain
$C_\Delta$ in the complex neighborhood. Its boundary is the
cycle swept over $\partial\Delta$, since the chosen cycle
has zero boundary. Stokes' theorem gives, with consistent orientations,
\[
 \int_{\partial\Delta}J_B(t)\,dt
 =\int_{\partial C_\Delta}\omega
 =\int_{C_\Delta}d\omega=0.
\]
Morera's theorem implies that $J_B$ is holomorphic throughout $B$, including at the critical values of $f$. The chains in this calculation remain disjoint from any singularities of the interior fibers.

Repeat the construction on increasing disks centered at $\tau$, choosing $\varepsilon$ and the lift $b$ separately on each disk. Every resulting function has germ $I_a$ at $\tau$, so the identity theorem gives agreement on overlaps. They define an entire function extending $I_a$.
\end{proof}

\begin{corollary}\label{cor:entire-vanishing}
For the resolved pencil of Section~\ref{sec:pencil}, suppose a critical fiber has exactly one critical point and this point is nondegenerate. Let $\delta$ be its vanishing cycle. Suppose $d$ is odd and $\lambda$ is a holomorphic $d$-form on $\Zcal$. If $I_\delta$ is nonzero and has finite branching over $U$, then it extends to an entire function.
\end{corollary}

\begin{proof}
For each $t\in U$, the restriction of $\lambda$ to $Z_t$
is a holomorphic $d$-form on a complex manifold of complex
dimension $d$, and is therefore closed. Lemma~\ref{lem:infinity-model} places the critical point away from $\Dcal$ and supplies the relative boundary over all finite parameters. Proposition~\ref{prop:semisimplicity} therefore gives $j\delta=0$, and Proposition~\ref{prop:boundary-entire} gives the extension.
\end{proof}

\section{Proof of the main theorems}\label{sec:main-proof}

\begin{proof}[Proof of Theorem~\ref{thm:section}]
Let $X$ be given by Lemma~\ref{lem:nash-boundary}, and make the translation \eqref{eq:origin}.
Remove from the set $\mathcal G$ of
Lemma~\ref{lem:good-real-directions} the directions for which
\eqref{eq:section-algebraic} specializes to the zero polynomial.
The excluded directions are the common zeros on $\Sph$ of
the coefficient polynomials of $Q$ when expanded in $t,w$.
Since $Q$ is not identically zero on $\Sph\times\R\times\R$,
this is a proper real algebraic subset and has empty interior.
The remaining set is dense, and we continue to denote it
by $\mathcal G$.
Fix $\xi\in\mathcal G$ and use the associated pencil
$f:\Mcal\to\C$ with regular base $U=\C\setminus\Sigma$.
The dependence of the pencil and of $\Sigma,U$ on $\xi$
is suppressed in the notation.

Choose oriented orthonormal coordinates $y_1,\ldots,y_{d+1}$ on $\xi^\perp$, extend them to $\R^n$ by orthogonal projection onto $\xi^\perp$, and then extend them complex-linearly to $\C^n$. Together with $t=\xi\cdot x$, they form linear coordinates on the ambient space. Define
\begin{equation}\label{eq:primitive}
 \Omega_\xi=dy_1\wedge\cdots\wedge dy_{d+1},
 \qquad
 \lambda_\xi=y_1\,dy_2\wedge\cdots\wedge dy_{d+1}.
\end{equation}
Thus $d\lambda_\xi=\Omega_\xi$. Pull these forms back along the composite $\Zcal\to M\setminus D_\infty\to X\subset\C^n$. They are holomorphic there, including along the affine exceptional set, since the map to the affine ambient space is regular. On every smooth complex $d$-dimensional fiber, $\lambda_\xi$ is closed.
We use the same notation for these forms and their pullbacks,
and take $\lambda=\lambda_\xi$ in the period construction of
Section~\ref{sec:semisimplicity}.

Let $t_+=b_K(\xi)$. For real $t<t_+$ sufficiently close to $t_+$, the entire section
\[
 C_{\xi,t}=K\cap\{x:\xi\cdot x=t\}
\]
lies in a small neighborhood of the unique support point. Indeed, otherwise a sequence of points outside that neighborhood would converge to another point in the support face. The real analytic Morse lemma gives coordinates $u_1,\ldots,u_{d+1}$ on the boundary near the support point in which the height is $t_+-\sum_{i=1}^{d+1}u_i^2$. At the support point, $dF\ne0$. The real and holomorphic implicit
function theorems identify $\partial K$ and $X_{\mathrm{reg}}$
locally with the real and complex graphs of the same analytic
function. The real analytic Morse coordinate change extends
holomorphically, and its nonzero Jacobian makes the extension
a local holomorphic coordinate change on $X_{\mathrm{reg}}$. For $t<t_+$ close enough, the entire boundary of $C_{\xi,t}$ is therefore the real sphere $\sum_i u_i^2=t_+-t$, with $u\in\R^{d+1}$, in this local complex fiber. It represents the vanishing cycle described in Section~\ref{sec:semisimplicity}. The resolution is an isomorphism on this neighborhood, so we denote the lifted sphere, with its boundary orientation, by $\delta_t$. Applying Stokes' theorem in the real section hyperplane and identifying its boundary with the lifted vanishing sphere gives
\begin{equation}\label{eq:real-period}
 A_K(\xi,t)=\int_{C_{\xi,t}}\Omega_\xi
 =\int_{\delta_t}\lambda_\xi>0.
\end{equation}
The specialized equation has positive degree in $w$, since otherwise a nonzero polynomial in $t$ would vanish on an interval. For the branching argument, let $B_\xi(t,w)$ be its monic square-free part in $w$ over $\C(t)$. Let $\mathcal E_\xi$ consist of its coefficient poles and the zeros, away from these poles, of $\Res_w(B_\xi,\partial_wB_\xi)$. This resultant is a nonzero rational function of $t$ by square-freeness, so $\mathcal E_\xi$ is finite; outside it the equation has only simple roots. Choose a real base point $\tau\in U\setminus\mathcal E_\xi$ in the above interval, and let $\delta$ be the homology class of $\delta_\tau$. The period $I_\delta$ is nonzero by \eqref{eq:real-period}.

The local triviality of the pair established after \eqref{eq:fibres} supplies homological transport along every path in $U$. The holomorphy of periods established in Section~\ref{sec:semisimplicity} then supplies analytic continuation of $I_\delta$ along each such path. Every returning germ still satisfies the original specialized
polynomial relation. Taking the square-free part over $\C(t)$
does not change its solutions as meromorphic germs, so each
returning germ also satisfies $B_\xi(t,w)=0$.
At $\tau$, this equation has only finitely many simple local
roots, so there are only finitely many such germs. Thus $I_\delta$ has finite branching over $U$. The set $\mathcal E_\xi$ is used only to choose the base point; it need not be removed from the geometric base $U$. Corollary~\ref{cor:entire-vanishing} now gives an entire extension.

By Lemma~\ref{lem:real-analytic-sections}, the actual section function is real analytic on the whole interval $I_K(\xi)$. Equation~\eqref{eq:real-period} and the identity theorem therefore give
\begin{equation}\label{eq:real-identification}
 A_K(\xi,t)=I_\delta(t)
 \qquad(t\in I_K(\xi)).
\end{equation}
This argument does not require transporting the actual real cycle through singular fibers of the complex algebraic family.

We next obtain a degree bound valid for all the directions under consideration. For this purpose, return to the original joint equation, rather than the square-free equation chosen for one direction. Write it in the form
\[
 Q(\xi,t,w)=\sum_{j=0}^{r}p_j(\xi,t)w^j,
 \qquad N_0=\max_{0\le j\le r}\deg_t p_j.
\]
Zero coefficients are omitted when taking the maximum defining $N_0$. This number is independent of $\xi$. By the identity theorem, the entire extension still satisfies $Q(\xi,t,I_\delta(t))=0$ for every $t\in\C$. For a fixed $\xi\in\mathcal G$, discard coefficients which vanish identically in $t$, and let $p_{r_\xi}(\xi,t)$ be the highest remaining coefficient in $w$. As noted above, $r_\xi\ge1$. The root bound for a polynomial gives, for all sufficiently large $|t|$,
\begin{equation}\label{eq:growth}
 |I_\delta(t)|
 \le 1+\max_{j<r_\xi}
 \left|\frac{p_j(\xi,t)}{p_{r_\xi}(\xi,t)}\right|
 \le C_\xi(1+|t|)^{N_0}.
\end{equation}
Here $C_\xi>0$ may depend on $\xi$. The leading coefficient has no zeros for sufficiently large $|t|$, and each quotient in \eqref{eq:growth} has growth at most $|t|^{N_0}$. Cauchy's estimates show that the entire function $I_\delta$ is a polynomial of degree at most $N_0$.

It follows from \eqref{eq:real-identification} that
\[
 \partial_t^j A_K(\xi,0)=0
 \qquad(\xi\in\mathcal G,\ j>N_0).
\]
For each fixed $j$, Lemma~\ref{lem:real-analytic-sections}
makes $\xi\mapsto\partial_t^jA_K(\xi,0)$ continuous.
Density of $\mathcal G$ therefore extends these identities
to every $\xi\in\Sph$. For each such $\xi$, real analyticity
in $t$ implies that $A_K(\xi,t)$ agrees near $0$ with its
Taylor polynomial of degree at most $N_0$. The identity
theorem extends this equality to all of $I_K(\xi)$.
Thus $K$ is polynomially integrable, with a uniform degree bound.
In particular, all central derivatives of every integer order
at least $N=\max\{n,N_0+1\}$ vanish in every direction.
Theorem~3.7 of \cite{KMY} now shows that $K$ is an ellipsoid.
\end{proof}

\begin{proof}[Proof of Theorem~\ref{thm:main}]
If $K$ is algebraically integrable, Lemma~\ref{lem:algebraization} supplies the hypotheses of Theorem~\ref{thm:section}. Hence $K$ is an ellipsoid.

Conversely, let
\[
 E=c+B^{1/2}B_2^n,
\]
where $c\in\R^n$ and $B$ is a positive definite real symmetric matrix.
Here $B^{1/2}$ denotes the positive definite symmetric square
root of $B$. Put $q(\xi)=\xi^TB\xi$. For internal sections,
\begin{equation}\label{eq:ellipsoid}
 A_E(\xi,t)=
 \frac{\omega_{n-1}\sqrt{\det B}}{\sqrt{q(\xi)}}
 \left(1-\frac{(t-c\cdot\xi)^2}{q(\xi)}\right)^m.
\end{equation}
Since $m=(n-1)/2$ is an integer, this is a polynomial in $t$ whose coefficients are algebraic functions of $\xi$. Integrating it from $c\cdot\xi-\sqrt{q(\xi)}$ to $t$ gives an algebraic cutoff volume. The substitution $\xi=\eta/|\eta|$, $t=s/|\eta|$ preserves algebraicity and shows that $E$ satisfies Definition~\ref{def:integrable}.
\end{proof}

\section{A special projection and the boundary at infinity}\label{sec:example}

The conclusion $j\delta=0$, taken by itself, does not imply that a finitely branched period is polynomial. The following example illustrates why the boundary hypothesis in Proposition~\ref{prop:boundary-entire} is needed.

\begin{proposition}\label{prop:special-projection}
There is a smooth affine surface with a polynomial projection having a nondegenerate critical point and a nonzero algebraic vanishing period which is not polynomial, although the vanishing cycle is zero in the first homology of a smooth projective compactification of a general fiber.
\end{proposition}

\begin{proof}
Consider the smooth graph
\[
 X_0=\{(u,v,w)\in\C^3:w=(1+u^2+v^2)^2\}
\]
and the projection $f_0=w$. In its global coordinates $(u,v)$,
\[
 f_0=1+2u^2+2v^2+(u^2+v^2)^2,
\]
so $(0,0)$ is a nondegenerate critical point with critical value $1$. The other critical locus is $1+u^2+v^2=0$, with critical value $0$. For real $t>1$ close to $1$, consider the disk in the real coordinate plane
\[
 \{(u,v)\in\R^2:u^2+v^2\le\sqrt t-1\}.
\]
Its boundary lies in $f_0^{-1}(t)$ and is the local vanishing sphere.
With $\lambda_0=u\,dv$, its boundary period is
\begin{equation}\label{eq:sqrt-example}
 I(t)=\pi(\sqrt t-1).
\end{equation}
It has two branches and is not a polynomial.

For $t\in\C\setminus\{0,1\}$, the affine fiber is the disjoint union of the two smooth conics
\[
 u^2+v^2=-1+\sqrt t,
 \qquad u^2+v^2=-1-\sqrt t.
\]
Each has smooth projective compactification $\Pj^1$, so the image of the vanishing cycle in compact first homology is zero. These compactifications are taken separately on the two
components, and each affine conic acquires two points at infinity. Around $t=0$, analytic continuation exchanges the two
components and induces a nontrivial permutation of their
four points at infinity. The local system formed by the
zeroth homology of these points therefore cannot extend
across $t=0$ as a locally constant system. Thus this special projection cannot have the relative boundary property in Proposition~\ref{prop:boundary-entire}. It does not contradict the generic-direction construction in Lemma~\ref{lem:infinity-model}.
\end{proof}

\begin{remark}
The odd-dimensional hypothesis in Theorem~\ref{thm:section} is essential. In the plane, the convex body
\[
 \{(x_1,x_2)\in\R^2:x_1^4+x_2^4\le1\}
\]
has smooth semialgebraic boundary and is not an ellipse. For $\xi=(\xi_1,\xi_2)\in S^1$, parametrize a line by
\[
 x=t\xi+u(-\xi_2,\xi_1).
\]
Substitution in $x_1^4+x_2^4=1$ gives a polynomial in $u$ with coefficients polynomial in $(\xi,t)$ and leading coefficient $\xi_1^4+\xi_2^4>0$. The chord length is the difference of its two real endpoint roots. To obtain a joint equation, take the monic polynomial whose roots are all ordered differences of the four complex roots. Its coefficients are symmetric polynomials in those roots, hence rational functions of the original coefficients. Clearing denominators gives a polynomial equation that is nontrivial on $S^1\times\R\times\R$ and is satisfied by every internal chord length. This does not give algebraicity of the cutoff area and does not contradict the even-dimensional obstruction of \cite{Vassiliev2015}.
\end{remark}

\par\medskip
\noindent\textbf{Declaration of AI use.}
The initial draft of this manuscript, including its mathematical
arguments, was generated entirely by ChatGPT.
The author has conducted a preliminary review and made revisions,
but has not independently verified all the proofs.
This manuscript is made available for the dissemination and discussion
of mathematical ideas and is not intended for submission to a journal.

\end{document}